\documentclass[11pt,reqno]{amsart}
\usepackage{epic}

\newcommand{\df}{\dfrac}
\newtheorem{theorem}{Theorem}[section]
\newtheorem{lemma}[theorem]{Lemma}
\newtheorem{corollary}[theorem]{Corollary}

\renewcommand{\a}{\alpha}

\begin{document}

\title[Partitions Associated With the Mock Theta Functions $\omega(q), \nu(q)$ and $\phi(q)$]{Partitions Associated With the Ramanujan/Watson Mock Theta Functions $\omega(q)$, $\nu(q)$ and $\phi(q)$} 

\author{George E. Andrews}
\address{Department of Mathematics, The Pennsylvania State University, University Park, PA 16802,
USA} \email{gea1@psu.edu}

\author{Atul Dixit}
\address{Department of Mathematics, Tulane University, New Orleans, LA, 70118, USA} 
\email{adixit@tulane.edu}

\author{Ae Ja Yee}
\address{Department of Mathematics, The Pennsylvania State University,  University Park, PA 16802,
USA} \email{auy2@psu.edu}


\footnotetext[2]{2000 AMS Classification Numbers: Primary, 11P81; Secondary, 05A17}

\footnotetext[3]{None of the authors have any competing interests in the manuscript.}

\maketitle

\begin{abstract}
The generating function of partitions with repeated (resp. distinct) parts such that each odd part is less than twice the smallest part is shown to be the third order mock theta function $\omega(q)$ (resp. $\nu(-q)$). Similar results for partitions with the corresponding restriction on each even part are also obtained, one of which involves the third order mock theta function $\phi(q)$. Congruences for the smallest parts functions associated to such partitions are obtained. Two analogues of the partition-theoretic interpretation of Euler's pentagonal theorem are also obtained.
\end{abstract}

\section{Introduction}
Partition-theoretic interpretations of various results involving mock theta functions have been the subject of intense study for many decades. For instance, Andrews and Garvan \cite{andrewsgarvan} reduced the proofs of ten identities for the fifth order mock theta functions given in Ramanujan's Lost Notebook \cite[p.~18--20]{lnb} to proving two conjectures based on the rank of a partition, and on the number of partitions of an integer with unique smallest part and all other parts less than or equal to the double (or one plus the double) of the smallest part. These conjectures, titled as Mock Theta Conjectures, were first proved by Hickerson \cite{hickerson}. 

A mock theta function itself may also admit a simple and interesting combinatorial interpretation. For example, consider $\chi_{1}(q)$, one of the fifth order mock theta functions of Ramanujan, defined by \cite[p.~278]{watson}
\begin{equation*}
\chi_1(q):=\sum_{n=0}^{\infty}\frac{q^n}{(q^{n+1};q)_{n+1}},
\end{equation*}
where $(a;q)_n:=\prod_{j=0}^{n-1}(1-aq^j)$.

It is easy to see that $q\chi_{1}(q)$ is the generating function for partitions in which no part is as large as twice the smallest part. Similarly another fifth order mock theta function \cite[p.~278]{watson}, namely 
\begin{equation*}
\chi_0(q):=\sum_{n=0}^{\infty}\frac{q^n}{(q^{n+1};q)_{n}},
\end{equation*}
can be interpreted \cite[Lemma 2]{geaglasgow} as the generating function for partitions with unique smallest part and the largest part at most twice the smallest part. There are other mock theta functions that admit combinatorial interpretations when different restrictions are put on parts of the corresponding partitions, for example the third order mock theta function $\psi(q)$ \cite[p.~57]{fine}, and the seventh order mock theta function $f(q)$ \cite[Theorem 4]{geasinica}.

In the present paper, among other things, we consider partitions of an integer in which each odd (even) part is less than (at most) twice the smallest part. For the odd case, the generating function surprisingly turns out to be $q\omega(q)$ (see Theorem \ref{qomegaq}), where $\omega(q)$ is the third order mock theta function due to Ramanujan and Watson defined by \cite{geabcblnb5}, \cite[p.~15]{lnb}, \cite[p.~62]{watson0}
\begin{align}\label{omegaq}
\omega(q):=\sum_{n=0}^{\infty} \frac{q^{2n^2+2n}}{(q;q^2)^2_{n+1}}. 
\end{align}
If we put an additional restriction that the parts be distinct, then the generating function of such partitions is $\nu(-q)$ (see Theorem \ref{nuqthm}), where $\nu(q)$ is another third order mock theta function defined by  \cite{geabcblnb5}, \cite[p.~31]{lnb}, \cite[p.~62]{watson0}
\begin{align}\label{nuq}
\nu(q):=\sum_{n=0}^{\infty} \frac{q^{n^2+n}}{(-q;q^2)_{n+1}}.
\end{align}
For partitions in which each even part is less than or equal to twice the smallest part, the generating function, as can be seen from Theorem \ref{rq}, involves the function $R(q)$ defined by
 \begin{align}\label{rqdef}
R(q):=\sum_{n=0}^{\infty} \frac{q^{\binom{n+1}{2}}}{(-q;q)_n}.
\end{align}
The function $R(q)$ appears in the Lost Notebook and was analyzed in \cite{geaepi}. Again restricting the parts in such partitions to be distinct leads us to a yet
 another third order mock theta function as a part of a representation for the generating function for these partitions (see Theorem \ref{phiqthm}). This mock theta function, namely
\begin{align}\label{phiq}
\phi(q):=\sum_{n=0}^{\infty} \frac{q^{n^2}}{(-q^2;q^2)_n},
\end{align}
itself does not have a simple partition-theoretic interpretation \cite[p.~58]{fine}.

Let $p_{\omega}(n)$ denote the number of partitions of $n$ in which each odd part is less than twice the smallest part. As discussed previously, 
\begin{equation*}
\sum_{n=1}^{\infty}p_{\omega}(n)q^n=q\omega(q).
\end{equation*}
Garthwaite and Penniston \cite{garpen} showed that the coefficients $a_{\omega}(n)$ of $\omega(q)$ satisfy infinitely many congruences of the similar type as Ramanujan's partition congruences. It is trivial to see that $p_{\omega}(n)=a_{\omega}(n-1)$. Waldherr \cite{waldherr} proved first explicit examples of such congruences, suggested by some computations done by Jeremy Lovejoy, which in terms of $p_{\omega}(n)$, can be written as
\begin{align*}
p_{\omega}(40n+28)&\equiv 0 \pmod{5},\\
p_{\omega}(40n+36)&\equiv 0 \pmod{5}.
\end{align*}
Let $p_{\nu}(n)$ denote the number of partitions of $n$ into distinct parts in which each odd part is less than twice the smallest part. 

The smallest parts function $\textup{spt}(n)$, counting the total number of appearances of the smallest parts in all partitions of $n$, has received great attention since it was introduced in \cite{gea266}. For generalizations and analogues of $\textup{spt}(n)$, we refer the reader to \cite{ack}, \cite{blo}, \cite{dixityee}, \cite{garvan}, \cite{cjs} and \cite{cjs1}.

If $\textup{spt}_{\omega}(n)$ and $\textup{spt}_{\nu}(n)$ denote the number of smallest parts in the partitions enumerated by $p_{\omega}(n)$ and $p_{\nu}(n)$ respectively, then we show that the following congruences hold:
{\allowdisplaybreaks\begin{align*}
\text{spt}_{\omega} (5n+3) &\equiv 0 \pmod{5},\nonumber\\
\text{spt}_{\omega} (10n+7) &\equiv 0 \pmod{5},\nonumber\\
\textup{spt}_{\omega} (10n+9) &\equiv 0 \pmod{5},\nonumber\\
\text{spt}_{\nu} (10n+8) &\equiv 0 \pmod{5}.
\end{align*}}
As shown in Section \ref{cong}, the first three congruences result from the following $q$-series identity:
\begin{align}\label{baileydiff1}
\sum_{n=1}^{\infty}\frac{(q;q)_nq^n}{(q;q^2)_n(1-q^n)^2}=\sum_{n=1}^{\infty}\frac{nq^n}{1-q^n}+\sum_{n=1}^{\infty}\frac{(-1)^n(1+q^{2n})q^{n(3n+1)}}{(1-q^{2n})^2}.
\end{align}
In fact, it was this identity that led us to investigate the aforementioned partitions and their connection with third order mock theta functions. This identity can be obtained by taking the second derivative of a special case of a ${}_{10}\phi_{9}$-transformation (see Equation \eqref{baitra} below) due to Bailey \cite[Equation (2.10)]{andrews1984}, \cite[Equation (6.3)]{bailey}. The details are given in Section \ref{cong}. 

Differentiatiation of identities in basic hypergeometric series with respect to a certain variable and then specializing them has been proved to be very useful in obtaining other important $q$-series identities \cite{geaonofri}, \cite{cdg}, in constructing new Bailey pairs \cite{geaillinois}. It also yields results with important partition-theoretic implications \cite{gea266}, \cite{agl}, \cite{dixityee}. Bailey's aforementioned ${}_{10}\phi_{9}$-transformation is no exception to this. 

We note that the series similar to the left-hand side of \eqref{baileydiff1}, namely
\begin{align*}
\sum_{n=1}^{\infty}\frac{(q;q)_nq^n}{(q^2;q^2)_n(1-q^n)^2},
\end{align*}
does not, however, seem to have a representation similar to that in \eqref{baileydiff1}. Also there do not seem to hold the corresponding congruences for the smallest parts function associated to the partitions in which each even part is less than or equal to twice the smallest part.

This paper is organized as follows. The preliminary results are provided in Section \ref{prelim}. In Section \ref{prp}, we consider partitions in which each odd (even) part is less than (at most) twice the smallest part, and in which repetition of parts is allowed. Such partitions where only the smallest part is not allowed to repeat are also studied. Section \ref{pdp} is devoted to studying partitions into distinct parts in which each odd (even) part is less than (at most) twice the smallest part. Two analogues of Euler's pentagonal number theorem \cite[p.~11, Corollary 1.7]{gea} are obtained in Section \ref{eulerpenta}. The congruences satisfied by the smallest parts functions associated with some of the partitions considered in Sections \ref{prp} and \ref{pdp} are proved in Section \ref{cong}. Finally in Section \ref{genmock}, we generalize two of the results obtained in the previous sections to those involving generalized third order mock theta functions.

\section{Preliminary results}\label{prelim}
We give below the standard results from the literature which will be used in the sequel. Throughout the paper, $q$ denotes a complex number such that $|q|<1$. Euler's identity \cite[p. 222, Equation (8.10.9)]{gasper}, which shows that the number of partitions into distinct parts equals that of partitions into odd parts, is expressed in terms of generating functions by
\begin{equation}\label{ei}
(-q;q)_{\infty}=\frac{1}{(q;q^2)_{\infty}},
\end{equation}
where $(a;q)_{\infty}:=\lim_{n\to\infty}(a;q)_{n}$. 

For $|z|<1$, the $q$-binomial theorem is given by \cite[p.~17, Equation (2.2.1)]{gea}
\begin{equation}\label{qbin}
\sum_{n=0}^{\infty}\frac{(a;q)_{n}z^n}{(q;q)_n}=\frac{(az;q)_{\infty}}{(z;q)_{\infty}}.
\end{equation}
Its special case $a=0$ gives another useful formula due to Euler \cite[p.~19, Equation (2.2.5)]{gea}:
\begin{equation}\label{euleri}
\sum_{n=0}^{\infty}\frac{z^n}{(q;q)_n}=\frac{1}{(z;q)_{\infty}}.
\end{equation}
Ramanujan's ${}_1\psi_{1}$ summation formula \cite[p.239, (II 29)]{gasper} is given by
\begin{equation}\label{1psi1sf}
\sum_{n=-\infty}^{\infty}\frac{(a;q)_{n}}{(b;q)_{n}}z^n=\frac{(az;q)_{\infty}(q/(az);q)_{\infty}(q;q)_{\infty}(b/a;q)_{\infty}}{(z;q)_{\infty}(b/(az);q)_{\infty}(b;q)_{\infty}(q/a;q)_{\infty}}.
\end{equation}
Let the Gaussian polynomial $\left[\begin{matrix} n\\m\end{matrix}\right]$ be defined by \cite[p.~35]{gea}
\begin{equation*}
\left[\begin{matrix} n\\m\end{matrix}\right]=\left[\begin{matrix} n\\m\end{matrix}\right]_{q}:=
\begin{cases}
(q;q)_{n}(q;q)_m^{-1}(q;q)_{n-m}^{-1},\hspace{2mm}\text{if}\hspace{1mm}0\leq m\leq n,\\
0,\hspace{2mm}\text{otherwise}.
\end{cases}
\end{equation*}
From \cite[p.~36, Equation (3.3.7)]{gea}, we have
\begin{align}
(z;q)_{N}^{-1}=\sum_{j=0}^{\infty}\left[\begin{matrix} N+j-1\\ j\end{matrix} \right]z^j\label{finpoc1},
\end{align}
and from \cite[p.~37, Equation (3.3.8)]{gea},
\begin{equation}\label{338}
\sum_{j=0}^{m}(-1)^j\left[\begin{matrix} m\\ j\end{matrix} \right]=
\begin{cases}
(q;q^2)_n, \hspace{2mm}\text{if}\hspace{1mm} m=2n,\\
0, \hspace{2mm}\text{if}\hspace{1mm} m\hspace{1mm}\text{is odd}.
\end{cases}
\end{equation}
Moreover \cite[p.~21, Corollary 2.7]{gea},
\begin{equation}\label{baicor}
\sum_{n=0}^{\infty}\frac{(a;q)_nq^{n(n+1)/2}}{(q;q)_n}=(-q;q)_{\infty}(aq;q^2)_{\infty}.
\end{equation}
Next, a $q$-analogue of Gauss' second theorem is given by \cite[Equation (1.8)]{gea55}
\begin{equation}\label{gst}
\sum_{n=0}^{\infty}\frac{(a;q)_n(b;q)_{n}q^{n(n+1)/2}}{(q;q)_n(qab;q^2)_n}=\frac{(-q;q)_{\infty}(aq;q^2)_{\infty}(bq;q^2)_{\infty}}{(qab;q^2)_{\infty}}.
\end{equation}
The following four-parameter $q$-series identity \cite[p. 141, Theorem 1]{gea90} will be frequently used in the proofs of our theorems:
\begin{align}
& \sum_{n=0}^{\infty} \frac{(B;q)_n (-Abq;q)_n q^n}{(-aq;q)_n (-bq;q)_n} \notag \\
&=\frac{-a^{-1} (B;q)_{\infty} (-Abq;q)_{\infty}}{(-bq;q)_{\infty} (-aq;q)_{\infty}} \sum_{m=0}^{\infty} \frac{(A^{-1};q)_m \left(\frac{Abq}{a}\right)^m}{\left(-\frac{B}{a};q\right)_{m+1}}\nonumber\\
&\quad+(1+b) \sum_{m=0}^{\infty} \frac{(-a^{-1};q)_{m+1} \left( -\frac{ABq}{a};q\right)_{m} (-b)^m}{\left(-\frac{B}{a};q\right)_{m+1} \left(\frac{Abq}{a};q\right)_{m+1}}. \label{gea90_thm1}
\end{align}
Finally, we note Bailey's ${}_{10}\phi_{9}$ transformation \cite[Equation (2.10)]{andrews1984}, \cite[Equation (6.3)]{bailey}:
\begin{multline}\label{baitra}
\lim_{N\to\infty}{}_{10}\phi_{9}\left(\begin{matrix} a,& q^2\sqrt{a},& -q^2\sqrt{a},& p_1, &
p_1q, &p_2, & p_2q, & f, & q^{-2N}, & q^{-2N+1}\\
 &\sqrt{a}, & -\sqrt{a},& \df{a
q^2}{p_1}, & \df{a q}{p_1}, & \df{a q^2}{p_2}, & \df{a q}{p_2}, & \df{aq^2}{f}, & aq^{2N+2}, & aq^{2N+1}
\end{matrix}\,; q^2,
 \df{a^3q^{4N+3}}{p_1^{2}p_2^{2}f}\right) \\
=\df{(a q;q)_{\infty}\left(\df{a q}{p_1p_2};q\right)_{\infty}}
{\left(\df{a q}{p_1};q\right)_{\infty}\left(\df{a q}{p_2};q\right)_{\infty}}\sum_{n=0}^{\infty}\frac{(p_1;q)_n(p_2;q)_n\left(\frac{aq}{f};q^2\right)_n}{(q;q)_n(aq;q^2)_n\left(\frac{aq}{f};q\right)_n}\left(\frac{aq}{p_1p_2}\right)^n,
\end{multline}
where
\begin{equation*}\label{bhs}
{}_r\phi_{s}\left(\begin{matrix} a_1, a_2, \ldots, a_r\\
  b_1,  b_2, \ldots, b_{s} \end{matrix}\,; q,
z \right) :=\sum_{n=0}^{\infty} \frac{(a_1;q)_n (a_2;q)_n \cdots (a_r;q)_n}{(q;q)_n (b_1;q)_n \cdots (b_{s};q)_n} z^n.
\end{equation*}
\section{ Partitions with repeated parts}\label{prp}

\begin{theorem}\label{qomegaq}
Let $\omega(q)$ be defined in \eqref{omegaq}. Then,
\begin{align}
\sum_{n=1}^{\infty} \frac{q^n}{(1-q^n)(q^{n+1};q)_n (q^{2n+2};q^2)_{\infty}}
=q\, \omega(q). \label{thm1}
\end{align}
\end{theorem}

\noindent{\bf Remark.} The series given in the theorem is clearly the generating function for partitions in which each odd part is less than twice the smallest part. 

\proof
First from \cite[p.62, Equation (26.88)]{fine},
\begin{align*}
\nu(q)+q\,\omega(q^2)=(-q^2;q^2)_{\infty}^3 (q^2;q^2)_{\infty}, 
\end{align*}
where $\nu(q)$ is defined in \eqref{nuq}.

Second,  in \eqref{gea90_thm1}, replace $q$ by $q^2$, then set $B=q^2, a=-b=q$,  and let $A\to 0$. This yields
\begin{align}
\sum_{n=0}^{\infty} \frac{(q^2;q^2)_n q^{2n}} {(q^6;q^4)_n}=\frac{-q^{-1} (q^2;q^2)_{\infty}}{(q^6;q^4)_{\infty}} \sum_{m=0}^{\infty} \frac{q^{m^2+m}}{(-q;q^2)_{m+1}} +(1-q) \sum_{m=0}^{\infty} \frac{(-q^{-1};q^2)_{m+1} q^m}{(-q;q^2)_{m+1}}.\label{5}
\end{align}

Finally replace $q$ by $q^2$ in \eqref{1psi1sf}, then set $a=-q, b=-q^3, z=q$ and simplify to obtain
\begin{align}
\sum_{m=0}^{\infty} \frac{q^{m}}{1+q^{2m+1}} =\frac{(q^4;q^4)_{\infty}^2}{(q^2;q^4)_{\infty}^2}.\label{1psi1}
\end{align}

We are now ready to prove the theorem. 

{\allowdisplaybreaks\begin{align}\label{string}
&\sum_{n=1}^{\infty} \frac{q^{2n}}{(1-q^{2n})(q^{2n+2};q^2)_n (q^{4n+4};q^4)_{\infty}}\nonumber\\
&=\sum_{n=0}^{\infty} \frac{q^{2n+2} (q^2;q^2)_{n}}{(q^{2};q^2)_{2n+2} (q^{4n+8};q^4)_{\infty}}\nonumber\\
&=\frac{1}{(q^4;q^4)_{\infty}} \sum_{n=0}^{\infty} \frac{q^{2n+2} (q^2;q^2)_{n}}{(1-q^2)(q^{6};q^4)_{n}}\nonumber\\
&=\frac{q^2}{(q^4;q^4)_{\infty} (1-q^2)} \sum_{n=0}^{\infty} \frac{q^{2n} (q^2;q^2)_{n}}{(q^{6};q^4)_{n}}\nonumber\\
&=\frac{q^2}{(q^4;q^4)_{\infty} (1-q^2)} \left( -q^{-1}(1-q^2)(q^4;q^4)_{\infty} \;\nu(q)+ (1-q) (1+q^{-1}) \sum_{m=0}^{\infty} \frac{q^{m}}{1+q^{2m+1}}\right)\nonumber\\
&=q\left(- \nu(q)+ \frac{1}{(q^4;q^4)_{\infty}} \sum_{m=0}^{\infty} \frac{q^{m}}{1+q^{2m+1}}\right)\nonumber\\
&=q\left(- \nu(q)+ \frac{(q^4;q^4)_{\infty}}{(q^2;q^4)_{\infty}^2} \right)\nonumber\\
&=q(-\nu(q)+ (q^4;q^4)_{\infty} (-q^2;q^2)_{\infty}^2)\nonumber\\
&=q(-\nu(q)+(-q^2;q^2)^3_{\infty} (q^2;q^2)_{\infty})\nonumber\\
&=q^2 \omega(q^2),
\end{align}}
where the fourth step follows from \eqref{5}, and the sixth and seventh steps follow from \eqref{1psi1} and \eqref{ei} respectively. 
This proves \eqref{thm1} with $q$ replaced by $q^2$, and thus completes the proof. 
\endproof

We now give another proof of the above theorem. We begin with a lemma that is also used in the subsequent sections.

\begin{lemma}\label{lemma1}
We have
\begin{align}
\sum_{n=0}^{\infty} \frac{q^{\frac{n}{2}}}{1+q^{\frac{1}{2}+n}}&=\sum_{n=0}^{\infty} \frac{q^n}{1-q^{4n+1}}-q^2\sum_{n=0}^{\infty} \frac{q^{3n}}{1-q^{4n+3}}=\sum_{n=0}^{\infty} \frac{  (-1)^n q^{\frac{n}{2}}}{ 1- q^{\frac{1}{2}+n} }.  \label{lem1} 
\end{align}
\end{lemma}

\proof
For obtaining the first equality, note that
{\allowdisplaybreaks\begin{align*}
\sum_{n=0}^{\infty} \frac{ q^{\frac{n}{2}}}{ 1+ q^{\frac{1}{2}+n} }&=\sum_{n=0}^{\infty} \frac{ q^{\frac{n}{2}}(1-q^{\frac{1}{2}+n})}{1-q^{1+2n}}\\
&=\sum_{n=0}^{\infty} \frac{q^n(1-q^{\frac{1}{2}+2n})}{1-q^{4n+1}} + \sum_{n=0}^{\infty} \frac{q^{\frac{1}{2}+n} (1-q^{\frac{3}{2}+2n})}{1-q^{4n+3}}\\
&=\sum_{n=0}^{\infty} \frac{q^n}{1-q^{4n+1}} -\sum_{n=0}^{\infty} \frac{q^{3n+2}}{1-q^{4n+3}} - q^{\frac{1}{2}}\left(\sum_{n=0}^{\infty} \frac{q^{3n}}{1-q^{4n+1}} -\sum_{n=0}^{\infty} \frac{q^n}{1-q^{4n+3}}\right)\\
&=\sum_{n=0}^{\infty} \frac{q^n}{1-q^{4n+1}} -\sum_{n=0}^{\infty} \frac{q^{3n+2}}{1-q^{4n+3}},
\end{align*}}
since 
\begin{align}
\sum_{n=0}^{\infty} \frac{q^{3n}}{1-q^{4n+1}} =\sum_{n=0}^{\infty}q^{3n}\sum_{m=0}^{\infty}q^{(4n+1)m}=\sum_{m=0}^{\infty}q^{m}\sum_{n=0}^{\infty}q^{(4m+3)n}=\sum_{n=0}^{\infty} \frac{q^{n}}{1-q^{4n+3}}. \label{lem4}
\end{align}
For the second statement, we employ the same method to see that
{\allowdisplaybreaks\begin{align*}
\sum_{n=0}^{\infty} \frac{  (-1)^n q^{\frac{n}{2}}}{ 1- q^{\frac{1}{2}+n} }&=\sum_{n=0}^{\infty} \frac{(-1)^n q^{\frac{n}{2}}(1+q^{\frac{1}{2}+n})}{1-q^{1+2n}}\\
&=\sum_{n=0}^{\infty} \frac{q^n(1+q^{\frac{1}{2}+2n})}{1-q^{4n+1}} -\sum_{n=0}^{\infty} \frac{q^{\frac{1}{2}+n} (1+q^{\frac{3}{2}+2n})}{1-q^{4n+3}}\\
&=\sum_{n=0}^{\infty} \frac{q^n}{1-q^{4n+1}} -\sum_{n=0}^{\infty} \frac{q^{3n+2}}{1-q^{4n+3}} + q^{\frac{1}{2}}\left(\sum_{n=0}^{\infty} \frac{q^{3n}}{1-q^{4n+1}} -\sum_{n=0}^{\infty} \frac{q^n}{1-q^{4n+3}}\right)\\
&=\sum_{n=0}^{\infty} \frac{q^n}{1-q^{4n+1}} -\sum_{n=0}^{\infty} \frac{q^{3n+2}}{1-q^{4n+3}},
\end{align*}}
where the last equality follows from \eqref{lem4}.
\endproof

\textbf{Second proof of Theorem \ref{qomegaq}.} First,

\begin{align}
&\sum_{n=1}^{\infty} \frac{q^n}{(q^n;q)_{n+1} (q^{2n+2};q^2)_{\infty}} =
\frac{q}{(q^2;q^2)_{\infty} (1-q)} \sum_{n=0}^{\infty} \frac{q^{n} (q;q)_{n}}{(q^3;q^2)_{n}}, \label{15}
\end{align}
as can be seen from the third step in \eqref{string}.

We now set $a=-b=A^{-1}=q^{1/2}$ and let $B\to 0$ in \eqref{gea90_thm1} to obtain
{\allowdisplaybreaks\begin{align}
&\sum_{n=1}^{\infty} \frac{q^{n} (q;q)_{n}}{(q^3;q^2)_{n}} \notag \\
&=\frac{-q^{-1/2} (q;q)_{\infty}}{(q^3;q^2)_{\infty}} \sum_{m=0}^{\infty} (q^{1/2};q)_m (-q^{1/2})^m +(1-q^{1/2}) \sum_{m=0}^{\infty} \frac{(-q^{-1/2};q)_{m+1} q^{m/2}}{(-q^{1/2};q)_{m+1}} \notag \\
&=\frac{-q^{-1/2} (q;q)_{\infty}}{(q^3;q^2)_{\infty}} \sum_{m=0}^{\infty} (q^{1/2};q)_m (-1)^m q^{m/2} +q^{-1/2} (1-q) \sum_{m=0}^{\infty} \frac{q^{m/2}}{1+ q^{1/2+m}} \notag\\
&=\frac{-q^{-1/2} (q;q)_{\infty}}{(q^3;q^2)_{\infty}} \sum_{m=0}^{\infty} (q^{1/2};q)_m (-1)^m q^{m/2} +q^{-1/2} (1-q) \left( \sum_{n=0}^{\infty} \frac{q^n}{1-q^{4n+1}}-q^2\sum_{n=0}^{\infty} \frac{q^{3n}}{1-q^{4n+3}} \right),  \label{16}
\end{align}}
where the last equality follows from Lemma \ref{lemma1}.

We now use \eqref{xi}, later proved in Theorem \ref{nuqthm} of Section \ref{pdp}:
\begin{align}
\sum_{m=0}^{\infty} (q;q^2)_m (-q)^m =(-q^2;q^2)_{\infty} \psi(q^2) -q\, \omega(q^2)=\frac{1}{(q^4;q^4)_{\infty}} \psi^2(q^2)-q\; \omega(q^2), \label{21}
\end{align}
where \cite[p. 23, Equation (2.2.13)]{gea}
\begin{align}\label{rampsi}
\psi(q)=\sum_{n=0}^{\infty} q^{\binom{n+1}{2}}=\frac{(q^2;q^2)_{\infty}}{(q;q^2)_{\infty}}.
\end{align}
The last equality in \eqref{21} follows from \eqref{ei}. 
From \eqref{15}, \eqref{16},  and \eqref{21},   we obtain 
\begin{align}\label{finid}
\sum_{n=1}^{\infty} \frac{q^n}{(q^n;q)_{n+1} (q^{2n+2};q^2)_{\infty}}  
&=q \; \omega(q) -\frac{q^{1/2}}{(q^2;q^2)_{\infty}} \left( \psi^2(q)- \sum_{n=0}^{\infty} \frac{q^n}{1-q^{4n+1}}+q^2\sum_{n=0}^{\infty} \frac{q^{3n}}{1-q^{4n+3}} \right) \notag \\
&=q \; \omega(q). 
\end{align}
This completes the proof.\\

\noindent{\bf Remark.}  As a by-product of \eqref{finid}, we obtain the identity
\begin{align*}
\psi^2({q})=\frac{(q^2;q^2)_{\infty}^2}{(q;q^2)_{\infty}^2}=\sum_{n=0}^{\infty} \frac{q^n}{1-q^{4n+1}}-q^2\sum_{n=0}^{\infty} \frac{q^{3n}}{1-q^{4n+3}}.
\end{align*}

We now obtain a representation corresponding to Theorem \ref{qomegaq} for the generating function of partitions in which each even part is at most twice the smallest part. 
\begin{theorem}\label{rq}
With $R(q)$ defined in \eqref{rqdef}, we have
\begin{align*}
\sum_{n=1}^{\infty} \frac{q^n}{(q^n;q)_{n+1} (q^{2n+1};q^2)_{\infty}}=-\frac{1}{2} R(q)+\frac{1}{(q;q^2)_{\infty}}\sum_{n=0}^{\infty} \frac{q^n}{1+q^n}.
\end{align*}
\end{theorem}

\begin{proof}
First we have
\begin{align}
\sum_{n=1}^{\infty} \frac{q^n}{(q^n;q)_{n+1} (q^{2n+1};q^2)_{\infty}}&=
\sum_{n=1}^{\infty} \frac{q^n(q;q)_{n-1}}{(q;q)_{2n} (q^{2n+1};q^2)_{\infty}}\notag \\
&=\frac{q}{(q;q^2)_{\infty} (1-q^2)} \sum_{n=0}^{\infty} \frac{q^{n} (q;q)_{n}}{(q^4;q^2)_{n}}. \label{2}
\end{align}
Set $b=-a$, $A=a^{-1}$ and let $B\to 0$ in \eqref{gea90_thm1} to obtain
\begin{align}
\sum_{n=0}^{\infty} \frac{q^{n} (q;q)_{n}}{(a^2q^2;q^2)_{n}}&=\frac{-a^{-1} (q;q)_{\infty}}{(a^2 q^2;q^2)_{\infty}} \sum_{m=0}^{\infty} (a;q)_m (-q/a)^m+(1-a) \sum_{m=0}^{\infty} \frac{(-a^{-1};q)_{m+1} a^{m}}{(-q/a;q)_{m+1}}\notag \\
&=\frac{-a^{-1} (q;q)_{\infty}}{(a^2q^2;q^2)_{\infty}} \sum_{m=0}^{\infty} (a;q)_m (-q/a)^m + a^{-1}(1-a^2)  \sum_{m=0}^{\infty} \frac{a^{m}}{1+q^{m+1}/a}. \label{1}
\end{align}
Note that
\begin{align*}
\sum_{m=0}^{\infty} (a;q)_m (-q/a)^m=\lim_{c\to 0} {_2}\phi_1 \left(\begin{matrix} a, q\\ c \end{matrix} ; q, -q/a\right).
\end{align*}
Heine's second transformation \cite[p.~241, Equation (III.2)]{gasper} of $_2 \phi_1$ gives
\begin{align}
\sum_{m=0}^{\infty} (a;q)_m (-q/a)^m&=\lim_{c\to 0} {_2}\phi_1 \left(\begin{matrix} a, q\\ c \end{matrix} ; q, -q/a\right) \notag \\
&=\lim_{c\to 0} \frac{(c/q;q)_{\infty} (-q^2/a;q)_{\infty}}{(c;q)_{\infty} (-q/a;q)_{\infty}} {_2} \phi_1 \left(\begin{matrix} -q^2/c, q\\ -q^2/a \end{matrix} ; q, c/q\right) \notag \\
&=\frac{(-q^2/a;q)_{\infty}}{(-q/a;q)_{\infty} } \sum_{n=0}^{\infty} \frac{q^{\binom{n+1}{2}}}{(-q^2/a;q)_n}. \label{11}
\end{align}
Substituting \eqref{11} into \eqref{1}, we see that
\begin{align*}
\sum_{n=0}^{\infty} \frac{q^{n} (q;q)_{n}}{(a^2q^2;q^2)_{n}} 
&=\frac{-a^{-1} (q;q)_{\infty} (-q^2/a;q)_{\infty}}{(a^2q^2;q^2)_{\infty} (-q/a;q)_{\infty}} \sum_{n=0}^{\infty} \frac{q^{\binom{n+1}{2}}}{(-q^2/a;q)_n} + a^{-1}(1-a^2) \sum_{m=0}^{\infty} \frac{a^m}{1+q^{m+1}/a},
\end{align*}
and when $a=q$, this becomes
\begin{align}
\sum_{n=0}^{\infty} \frac{q^{n} (q;q)_{n}}{(q^4;q^2)_{n}}=\frac{-q^{-1} (q;q)_{\infty}}{2 (q^4;q^2)_{\infty} }\sum_{n=0}^{\infty} \frac{q^{\binom{n+1}{2}}}{(-q;q)_n} + q^{-1} (1-q^2) \sum_{n=0}^{\infty} \frac{q^n}{1+q^{n}}.  \label{12}
\end{align}
Now substitute \eqref{12} into \eqref{2} to obtain
\begin{align*}
\sum_{n=1}^{\infty} \frac{q^n}{(q^n;q)_{n+1} (q^{2n+1};q^2)_{\infty}}&=-\frac{1}{2} \sum_{n=0}^{\infty} \frac{q^{\binom{n+1}{2}}}{(-q;q)_n} + \frac{1}{(q;q^2)_{\infty}} \sum_{n=0}^{\infty} \frac{q^n}{1+q^{n}}.
\end{align*}
This completes the proof.

\end{proof}

\subsection{Aforementioned partitions with unique smallest part}\label{unique}
In this section, we consider the aforementioned partitions with the only additional restriction being that the smallest part cannot repeat. 
\begin{theorem} \label{thm0.2}
The following is true:
\begin{align*}
\sum_{n=1}^{\infty} \frac{q^n}{(q^{n+1};q)_{n} (q^{2n+1};q^2)_{\infty}}=-1+(-q;q)_{\infty}.
\end{align*}
\end{theorem}

\begin{proof}
First note that
\begin{align}
\sum_{n=0}^{\infty} \frac{q^n}{(q^{n+1};q)_{n} (q^{2n+1};q^2)_{\infty}}&=
\frac{1}{(q;q^2)_{\infty} } \sum_{n=0}^{\infty} \frac{q^{n} (q;q)_{n}}{(q^2;q^2)_{n}}. \label{13}
\end{align}
Set $a=-b= A=-1$ and let $B\to 0$ in \eqref{gea90_thm1} to see that
\begin{align}
\sum_{n=0}^{\infty} \frac{q^{n} (q;q)_{n}}{(q^2;q^2)_{n}}&=\frac{ (q;q)_{\infty}}{(q^2;q^2)_{\infty}} \sum_{m=0}^{\infty} (-1;q)_m q^m. \label{14}
\end{align}
From \eqref{13} and \eqref{14},
\begin{align*}
\sum_{n=0}^{\infty} \frac{q^n}{(q^{n+1};q)_{n} (q^{2n+1};q^2)_{\infty}}&=
\sum_{m=0}^{\infty} (-1;q)_m q^m\\
&=-1+2\big(1+\sum_{m=1}^{\infty} (-q;q)_{m-1} q^{m}\big)\\
&=-1+2(-q;q)_{\infty},
\end{align*}
where the last step is valid since
\begin{equation}\label{seplarge}
1+\sum_{m=1}^{\infty} (-q;q)_{m-1} q^{m}=(-q;q)_{\infty},
\end{equation}
which in turn follows from the fact that both sides represent the generating function for partitions into distinct parts, with the left one doing so by separating out the largest part.

Hence,
\begin{align*}
\sum_{n=1}^{\infty} \frac{q^n}{(q^{n+1};q)_{n} (q^{2n+1};q^2)_{\infty}}&=\sum_{n=0}^{\infty} \frac{q^n}{(q^{n+1};q)_{n} (q^{2n+1};q^2)_{\infty}}-\frac{1}{(q;q^2)_{\infty}}\\
&=\sum_{n=0}^{\infty} \frac{q^n}{(q^{n+1};q)_{n} (q^{2n+1};q^2)_{\infty}} -(-q;q)_{\infty}\\
&=-1+(-q;q)_{\infty}.
\end{align*}%
\end{proof}

\textbf{A combinatorial proof of Theorem~\ref{thm0.2}.} Theorem~\ref{thm0.2} yields the following partition theorem. We provide a combinatorial proof. 
\begin{theorem} \label{thm0.5}
Let $n$ be a positive integer. Then 
the number of partitions of $n$ with unique smallest part in which each even does not exceed twice the smallest part equals the number of partitions of $n$ into distinct parts. 
\end{theorem}

\proof
For any fixed $k$, let $A_k$ be the set of partitions $\lambda$ in which $k$ is the unique smallest part and even parts are $\le 2k$. In other words,
\begin{align*}
A_k=\{ \lambda =(k, (k+1)^{f_{k+1}},(k+2)^{f_{k+2}}, \ldots)  \; |\;  f_{n}=0 \text{ if $n$ is even and greater than $2k$} \}.
\end{align*}
We now write $f_n$ as a binary representation:
\begin{align*}
f_n=\sum_{i=0}^{\infty} a_{n,i} 2^i,
\end{align*} 
where $a_{n,i}$ is either $0$ or $1$.  Note that this binary representation is unique. Also, for $n_1, n_2>k$ with the constraint that they do not exceed $2k$ if they are even,  we have 
\begin{align*}
2^in_1=2^j n_2  \quad \text{if and only if } \quad  i=j, n_1=n_2.
\end{align*}
The necessity is trivial. The sufficiency is clear as well for $n_1\equiv n_2 \equiv 1 \pmod{2}$. We now suppose that $n_1=2^{r_1}(2a+1)$ and $n_2=2^{r_2}(2b+1)$ with $r_1\ge r_2$, and also suppose that 
\begin{align*}
2^{i+r_1}(2a+1)=2^{j+r_2}(2b+1).
\end{align*} 
Then $i+r_1=j+r_2$ and $a=b$. If $r_1>r_2$, then  since $n_1=2^{r_1} (2a+1) \le 2k$,
\begin{align*}
n_2=2^{r_2}(2b+1)\le 2^{r_1-1}(2a+1)  \le k, 
\end{align*}
which is a contradiction.  
Thus we can say that $A_k$ is the set of partitions into distinct parts of the form $2^i n$ with $i\ge 0$, $n>k$ for $n$ odd and $k<n\le 2k$ for $n$ even. 

We now prove the theorem by showing that any integer $>k$ can be uniquely written as an integer of the form $2^i n$,  $i\ge 0$, $n>k$ for $n$ odd, and $k< n\le 2k$ for $n$ even.  

Let $k$ be even, i.e., $k=2m$ for some $m\ge 1$.  For an integer $N>k$, we now write it as
\begin{align*}
N=2^i M,
\end{align*}
where $M$ is odd. If $M>2m$, then
\begin{align*}
N=2^i \big (2m+(M-2m)\big),
\end{align*}
which is of the desired form. If $M<2m$, then since $N=2^i M> 2m$,  there exists a unique $j$ such that $2m< 2^j M \le 4m$, so we write
\begin{align*}
N=2^{i-j}( 2^jM),
\end{align*}
as is desired. 

We can prove the result for odd $k$ in a similar way, and hence the proof is omitted.
 \endproof
 
\noindent{\bf Remark.} Theorem~\ref{thm0.2} can be proved as follows without using \eqref{gea90_thm1}.
{\allowdisplaybreaks\begin{align*}
\sum_{n=1}^{\infty} \frac{q^n}{(q^{n+1};q)_{n} (q^{2n+1};q^2)_{\infty}}
&= \sum_{n=1}^{\infty} \frac{q^{n} (q^{2n+2};q^2)_{\infty}}{(q^{n+1};q)_{n} (q^{2n+1};q)_{\infty}}\\
&= \sum_{n=1}^{\infty} \frac{q^{n} (q^{2n+2};q^2)_{n} (-q^{2n+1};q)_{\infty}}{(q^{n+1};q)_{n}}\\
&=\sum_{n=1}^{\infty} q^{n} (-q^{n+1};q)_{n} (-q^{2n+1};q)_{\infty} \\
&=\sum_{n=1}^{\infty} q^n (-q^{n+1};q)_{\infty}\\
&=-1+(-q;q)_{\infty},
\end{align*}}
where the last equality holds because if we interpret the $n$ in the sum $\sum_{n=1}^{\infty} q^n (-q^{n+1};q)_{\infty}$ as denoting the smallest part in a partition, we obtain partitions into distinct parts.

We now give an analogue of Theorem \ref{thm0.2} for partitions with unique smallest part and all even parts at most twice the smallest part.
\begin{theorem}
We have
\begin{align*}
\sum_{n=1}^{\infty} \frac{q^n}{(q^{n+1};q)_{n} (q^{2n+2};q^2)_{\infty}}&=  q^2\sum_{m=0}^{\infty} \frac{ q^{3m}}{(q;q^2)_{m+1}}  + \frac{q}{(q^2;q^2)_{\infty}}.
\end{align*}
\end{theorem}

\begin{proof}
Note that
\begin{align}
&\sum_{n=1}^{\infty} \frac{q^{n}}{(q^{n+1};q)_{n} (q^{2n+2};q^2)_{\infty}} 
=\frac{1}{(q^2;q^2)_{\infty} } \sum_{n=1}^{\infty} \frac{(q;q)_{n} q^n}{(q;q^2)_{n}}. \label{20}
\end{align}
Set $a=-b=-q^{1/2}$, $B=q^2$, and let $A\to 0$ in \eqref{gea90_thm1} to see that
\begin{align*}
\sum_{n=0}^{\infty} \frac{(q^2;q)_{n} q^n}{(q^3;q^2)_n}&=\frac{q^{-\frac{1}{2}} (q^2;q)_{\infty}}{(q^3;q^2)_{\infty}} \sum_{m=0}^{\infty} \frac{q^{\binom{m+1}{2}}}{(q^{\frac{3}{2}};q)_{m+1}}  +(1+ q^{\frac{1}{2}}) \sum_{m=0}^{\infty} \frac{(q^{-\frac{1}{2}};q)_{m+1} (-1)^m q^{\frac{m}{2}}}{(q^{\frac{3}{2}};q)_{m+1}} \notag\\
&=\frac{q^{-\frac{1}{2}} (q;q)_{\infty}}{(q;q^2)_{\infty}} \sum_{m=0}^{\infty} \frac{q^{\binom{m+1}{2}}}{(q^{\frac{3}{2}};q)_{m+1}} + (1-q) (1-q^{-\frac{1}{2}}) \sum_{m=0}^{\infty} \frac{  (-1)^m q^{\frac{m}{2}}}{(1-q^{\frac{1}{2}+m})(1- q^{\frac{3}{2}+m}) }.
\end{align*}
Multiplying both sides by $q$, we obtain
\begin{align}
\sum_{n=1}^{\infty} \frac{(q;q)_{n} q^n}{(q;q^2)_n}&=
\frac{q^{\frac{1}{2}} (q;q)_{\infty}}{(q;q^2)_{\infty}} \sum_{m=0}^{\infty} \frac{q^{\binom{m+1}{2}}}{(q^{\frac{3}{2}};q)_{m+1}} - q^{\frac{1}{2}} (1-q^{\frac{1}{2}}) (1-q) \sum_{m=0}^{\infty} \frac{ (-1)^m q^{\frac{m}{2}}}{(1-q^{\frac{1}{2}+m})(1- q^{\frac{3}{2}+m}) }.   \label{17}
\end{align}
Also, using \eqref{finpoc1} in the first step below, then \eqref{gst} with $a=q^{j+1}$ and $b=0$ in the third, and then \eqref{euleri} in the fifth, we see that
\begin{align}
\sum_{n=0}^{\infty} \frac{q^{\binom{n+1}{2}}}{(q^{\frac{3}{2}};q)_{n+1}} &=\sum_{n=0}^{\infty} {q^{\binom{n+1}{2}}}\sum_{j=0}^{\infty} \left[\begin{matrix} n+j\\ j\end{matrix} \right] q^{\frac{3}{2}j }\notag \\
&=\sum_{j=0}^{\infty} q^{\frac{3}{2}j } \sum_{n=0}^{\infty} \frac{(q^{j+1};q)_n}{(q;q)_n} q^{\binom{n+1}{2}} \notag\\
&=\sum_{j=0}^{\infty} q^{\frac{3}{2}j } (-q;q)_{\infty} (q^{j+2};q^2)_{\infty} \notag\\
&=(-q;q)_{\infty} (q^2;q^2)_{\infty} \sum_{j=0}^{\infty} \frac{q^{3j}}{(q^2;q^2)_{j}} + q^{\frac{3}{2}}(-q;q)_{\infty} (q^3;q^2)_{\infty} \sum_{j=0}^{\infty} \frac{q^{3j}}{(q^3;q^2)_{j}} \notag\\
&=\frac{(-q;q)_{\infty} (q^2;q^2)_{\infty} }{(q^3;q^2)_{\infty}} +q^{\frac{3}{2}} (-q;q)_{\infty} (q;q^2)_{\infty} \sum_{j=0}^{\infty} \frac{q^{3j}}{(q;q^2)_{j+1}} \notag \\
&=\frac{(-q;q)_{\infty} (q^2;q^2)_{\infty} }{(q^3;q^2)_{\infty}} +q^{\frac{3}{2}}\sum_{j=0}^{\infty} \frac{q^{3j}}{(q;q^2)_{j+1}}, \label{18}
\end{align}
and 
{\allowdisplaybreaks\begin{align}
(1-q) \sum_{m=0}^{\infty} \frac{ (-1)^m q^{\frac{m}{2}}}{(1- q^{\frac{1}{2}+m})(1-q^{\frac{3}{2}+m}) } &=
\sum_{m=0}^{\infty} \frac{(-1)^m q^{\frac{m}{2}}}{1-q^{\frac{1}{2}+m}} +\sum_{m=0}^{\infty} \frac{(-1)^{m+1} q^{1+\frac{m}{2}}}{1-q^{\frac{3}{2}+m}} \notag \\
&=(1+q^{\frac{1}{2}})\sum_{m=0}^{\infty} \frac{(-1)^m q^{\frac{m}{2}}}{1-q^{\frac{1}{2}+m}} -\frac{q^{\frac{1}{2}}}{1-q^{\frac{1}{2}}}  \notag\\
&= (1+q^{\frac{1}{2}}) \left( \sum_{n=0}^{\infty} \frac{q^n}{1-q^{4n+1}} - q^2\sum_{n=0}^{\infty} \frac{q^{3n}}{1-q^{4n+3}} \right) -\frac{q^{\frac{1}{2}}}{1-q^{\frac{1}{2}}},  \label{19}
\end{align}}
where the last step follows from \eqref{lem1}.  
Thus, by \eqref{17}, \eqref{18}, \eqref{19}, and \eqref{20},  we obtain
{\allowdisplaybreaks\begin{align*}
&\sum_{n=1}^{\infty} \frac{q^{n}}{(q^{n+1};q)_{n} (q^{2n+2};q^2)_{\infty}} \notag \\
&=\frac{q^{\frac{1}{2}}  (q^2;q^2)_{\infty}}{(q;q^2)_{\infty} (q^3;q^2)_{\infty} } + q^{2} \sum_{j=0}^{\infty} \frac{ q^{3j}}{(q;q^2)_{j+1}} - \frac{q^{\frac{1}{2}}(1-q)}{(q^2;q^2)_{\infty} } \left( \sum_{n=0}^{\infty} \frac{q^n}{1-q^{4n+1}} - q^2\sum_{n=0}^{\infty} \frac{q^{3n}}{1-q^{4n+3}} \right) +\frac{q}{(q^2;q^2)_{\infty}}
\\
&=q^{2} \sum_{j=0}^{\infty} \frac{ q^{3j}}{(q;q^2)_{j+1}}  +\frac{q}{(q^2;q^2)_{\infty}} +\frac{q^{\frac{1}{2}}(1-q)}{(q^2;q^2)_{\infty}} \left(  \frac{(q^2;q^2)^2_{\infty}}{(q;q^2)_{\infty}^2}- \sum_{n=0}^{\infty} \frac{q^n}{1-q^{4n+1}} + q^2\sum_{n=0}^{\infty} \frac{q^{3n}}{1-q^{4n+3}} \right)\\
&=q^{2} \sum_{j=0}^{\infty} \frac{ q^{3j}}{(q;q^2)_{j+1}}  +\frac{q}{(q^2;q^2)_{\infty}},
\end{align*}}
where we used \eqref{lem1}, and \eqref{1psi1} with $q$ replaced by $q^{1/2}$ in the last step.
\end{proof}

\section{Partitions with distinct parts}\label{pdp}
In the theorem below, we show that the generating function for partitions into distinct parts where each odd is less than twice the smallest part is the third order mock theta function $\nu(-q)$.
\begin{theorem}\label{nuqthm}
Let $\psi(q)$ be defined in \eqref{rampsi}, and $\omega(q)$ and $\nu(q)$ be defined in \eqref{omegaq} and \eqref{nuq} respectively. Then,
\begin{align}
\sum_{n=0}^{\infty} q^n (-q^{n+1};q)_n (-q^{2n+2};q^2)_{\infty}&= q \; \omega(q^2) +(-q^2;q^2)_{\infty} \psi(q^2)  \label{thm1.1}\\
&=\nu(-q). \notag
\end{align}

\end{theorem}

\proof
 If we divide both sides of the first equality in \eqref{thm1.1} by $(-q^2;q^2)_{\infty}$, we see that proving it is equivalent to showing
 \begin{align*}
 \sum_{n=0}^{\infty} \frac{q^n(-q;q^2)_n}{(-q;q)_n}=q\frac{\omega (q^2)}{(-q^2;q^2)_{\infty}} +\psi(q^2).
 \end{align*}
 
We now apply \eqref{gea90_thm1} with $A=B/aq$, $a=1$, and then let $b\to 0$.  This yields
\begin{align}
\sum_{n=0}^{\infty} \frac{(B^2;q^2)_n q^n}{(-q;q)_n}=-\frac{(B^2;q^2)_\infty}{(-q;q)_{\infty}} \sum_{m=0}^{\infty} \frac{B^m}{(-B;q)_{m+1}} +\sum_{m=0}^{\infty} \frac{(-1;q)_{m+1} (-1)^m q^{\binom{m}{2}} B^{2m}}{(B^2;q^2)_{m+1}}. \label{iii}
\end{align} 
Next from \eqref{finpoc1},
\begin{align}
\sum_{m=0}^{\infty} \frac{B^m}{(-B;q)_{m+1}} &=\sum_{m,n=0}^{\infty} B^{m+n} (-1)^n \left[ \begin{matrix} m+ n\\n \end{matrix} \right] \notag\\
&=\sum_{N=0}^{\infty} B^N \sum_{n=0}^{N} (-1)^n \left[ \begin{matrix} N\\n \end{matrix} \right] \notag\\
&=\sum_{N=0}^{\infty} B^{2N} \sum_{n=0}^{2N} (-1)^n \left[ \begin{matrix} 2N\\n \end{matrix} \right] \notag \\
&=\sum_{N=0}^{\infty} B^{2N} (q;q^2)_N, \label{iv}
\end{align}
where in the penultimate as well as in the ultimate step, we used \eqref{338}.

Substituting \eqref{iv} into \eqref{iii}, we find
\begin{align}
\sum_{n=0}^{\infty} \frac{(B^2;q^2)_n q^n}{(-q;q)_n} &=-\frac{(B^2;q^2)_{\infty}}{(-q;q)_{\infty}} \sum_{n=0}^{\infty} B^{2n} (q;q^2)_{n} +2\sum_{m=0}^{\infty} \frac{(-q;q)_m (-1)^m q^{\binom{m}{2}} B^{2n}} {(B^2;q^2)_{m+1}}.  \label{v}
\end{align}
Now set $B^2=-q$ (i.e., $B= iq^{1/2}$) in \eqref{v} to deduce that
\begin{align}
\sum_{m=0}^{\infty} \frac{(-q;q^2)_m q^m}{(-q;q)_m}&=-\frac{(-q;q^2)_{\infty}}{(-q;q)_{\infty}} \sum_{m=0}^{\infty} (-q)^m (q;q^2)_m +2\sum_{m=0}^{\infty} \frac{(-q;q)_m q^{\binom{m+1}{2}}}{(-q;q^2)_{m+1}}\notag \\
&=:\frac{-1}{(-q^2;q^2)_{\infty}} S_1(q) +2S_2(q). \label{vi}
\end{align}

We now evaluate $S_2(q)$:
\begin{align}
S_2(q)&=\sum_{m=0}^{\infty} \frac{q^{\binom{m+1}{2}}}{(q;q)_m} \frac{(-q;q)_m (q;q)_m}{(-q;q^2)_{m+1}} \notag\\
&= \frac{(-q;q)_{\infty} (-q^2;q^2)_{\infty} (q^2;q^2)_{\infty}}{(1+q)(-q^3;q^2)_{\infty}}\notag \\
&=\frac{(q^4;q^4)_{\infty}}{(q^2;q^4)_{\infty}} \notag \\
&=\psi(q^2), \label{vii}
\end{align}
where we used \eqref{gst} in the first step and \eqref{rampsi} in the last.
Note that from \cite[p.~29, Exercise 6]{gea},
\begin{equation}\label{p29gea}
\sum_{m=0}^{\infty}
\frac{q^{m^2}x^m}{(y;q^2)_{m+1}}=\sum_{m=0}^{\infty}(-xq/y;q^2)_{m}y^m.
\end{equation}
Using the above identity with $x=q$ and $y=-q$ in the first step below, and the $q$-binomial theorem \eqref{qbin} with $q$ replaced by $q^2, z=-q$ and $a=q^{2m+2}$ in the next step, we observe that
{\allowdisplaybreaks\begin{align}
S_1(q)&=\sum_{m=0}^{\infty} (-q)^m (q;q^2)_m\notag \\
&=\sum_{m=0}^{\infty} \frac{q^{m^2+m}}{(-q;q^2)_{m+1}}\notag \\
&=\sum_{m=0}^{\infty} q^{m^2+m} \sum_{n=0}^{\infty} (-1)^n q^n \left[ \begin{matrix} n+ m\\m \end{matrix} \right]_{q^2}\notag \\
&=\sum_{m=0}^{\infty} q^{m^2+m} \sum_{n=0}^{\infty} q^{2n} \left[ \begin{matrix} 2n+ m\\m \end{matrix} \right]_{q^2} -\sum_{m=0}^{\infty} q^{m^2+m} \sum_{n=0}^{\infty} q^{2n+1} \left[ \begin{matrix} 2n+1+ m\\m \end{matrix} \right]_{q^2}\notag \\
&=:T_1(q^2)-qT_2(q^2). \label{viii}
\end{align}}

Also,
{\allowdisplaybreaks\begin{align}
T_1(q)&=\sum_{n=0}^{\infty} q^n \sum_{m=0}^{\infty} \frac{q^{\binom{m+1}{2}}(q^{2n+1};q)_m}{(q;q)_m}\notag \\
&=\sum_{n=0}^{\infty} q^{n} (q^{2n+2};q^2)_{\infty} (-q;q)_{\infty}\notag \\
&=(-q;q)_{\infty} (q^2;q^2)_{\infty}\sum_{n=0}^{\infty} \frac{q^n}{(q^2;q^2)_n}\notag \\
&=\frac{(-q;q)_{\infty} (q^2;q^2)_{\infty}}{(q;q^2)_{\infty}}\notag \\
&=(-q;q)_{\infty} \psi(q), \label{ix}
\end{align}}%
where \eqref{baicor} with $a=q^{2n+1}$ was used in the second step. Again using \eqref{baicor}, this time with $a=q^{2n+2}$, in the second step below, we find that
{\allowdisplaybreaks\begin{align}
T_2(q)&=\sum_{n=0}^{\infty} q^{n} \sum_{m=0}^{\infty} \frac{q^{\binom{m+1}{2}}(q^{2n+2};q)_{m}}{(q;q)_m}\notag \\
&=\sum_{n=0}^{\infty} q^n (q^{2n+3};q^2)_n (-q;q)_{\infty}\notag \\
&=(-q;q)_{\infty} (q;q^2)_{\infty} \sum_{n=0}^{\infty} \frac{q^n}{(q;q^2)_{n+1}}\notag \\
&=\sum_{n=0}^{\infty} \frac{q^n}{(q;q^2)_{n+1}}\notag \\
&=\omega(q), \label{x}
\end{align}}
where the last step follows from \cite[p.~61, Equation (2.6.84)]{fine}.

Substituting \eqref{ix}  and \eqref{x} into \eqref{viii},  we find that
\begin{align}
S_{1}(q)&=(-q^2;q^2)_{\infty} \psi(q^2)-q\; \omega(q^2). \label{xi}
\end{align}
Now substitute \eqref{xi} and \eqref{vii} in \eqref{vi} to finally deduce
{\allowdisplaybreaks\begin{align}\label{nufinal}
\sum_{n=0}^{\infty} \frac{(-q;q^2)_nq^n}{(-q;q)_n}& =\frac{-1}{(-q^2;q^2)_{\infty}} \left((-q^2;q^2)_{\infty}\psi(q^2)-q\omega(q^2)\right) +2\psi(q^2)\nonumber\\
&=q \frac{\omega(q^2)}{(-q^2;q^2)_{\infty}} +\psi(q^2)\nonumber\\
&=\frac{\nu(-q)}{(-q^2;q^2)_{\infty}},
\end{align}}
where the last step results from \cite[p.~62, Equation (26.88)]{fine}.
\endproof

The result corresponding to Theorem \ref{nuqthm} for partitions into distinct parts where each even part is at most twice the smallest part is presented below.
\begin{theorem}\label{phiqthm}
Let $\phi(q)$ be defined in \eqref{phiq}. Then,
\begin{align*}
1+ q \sum_{n=0}^{\infty} q^n (-q^{n+1};q)_n (-q^{2n+1};q^2)_{\infty} = \frac{1-\phi(q)}{(-q;q^2)_{\infty}} +(q^2;q^2)_{\infty} (-q;q^2)_{\infty}^2.
\end{align*}
\end{theorem}

\begin{proof}

In \eqref{gea90_thm1}, we set $A=B/bq$ and $a=1$, and then let $b\to 0$. Then we obtain
\begin{align*}
\sum_{n=0}^{\infty} \frac{(B^2;q^2)_n q^n}{(-q;q)_n}= -\frac{(B^2;q^2)_{\infty}}{(-q;q)_{\infty}} \sum_{n=0}^{\infty} B^{2n}  (q;q^2)_n +2\sum_{m=0}^{\infty} \frac{(-q;q)_m (-1)^m q^{\binom{m}{2}} B^{2m}}{(B^2;q^2)_{m+1}}. 
\end{align*}
We now set $B=qi$, and we obtain
{\allowdisplaybreaks\begin{align*}
\sum_{n=0}^{\infty} \frac{(-q^2;q^2)_n q^n}{(-q;q)_n} &= -\frac{(-q^2;q^2)_{\infty}}{(-q;q)_{\infty}} \sum_{n=0}^{\infty} (-1)^n q^{2n}  (q;q^2)_n +2\sum_{m=0}^{\infty} \frac{(-q;q)_m  q^{\binom{m}{2}}  q^{2m}}{(-q^2;q^2)_{m+1}}\\
&=\frac{-1}{(-q;q^2)_{\infty}} \sum_{n=0}^{\infty} (-1)^n q^{2n}  (q;q^2)_n +2\sum_{m=0}^{\infty} \frac{(-q;q)_m  q^{\binom{m}{2}}  q^{2m}}{(-q^2;q^2)_{m+1}}\\
&=: \frac{-1}{(-q;q^2)_{\infty}} U_1(q)+ 2U_2(q). 
\end{align*}}
We now claim that
\begin{align*}
U_1(q)&=q^{-1} (\phi(q) -1),\\
U_2(q)&=\sum_{n=0}^{\infty} q^{n^2+2n} = q^{-1} \sum_{n=1}^{\infty} q^{n^2}.
\end{align*}

First, for $U_1(q)$, set $x=q^2, y=-q^2$ in \eqref{p29gea}. Then,
\begin{align*}
U_1(q)=\sum_{m=0}^{\infty} (q;q^2)_m (-q^2)^m =\sum_{m=0}^{\infty} \frac{q^{m^2+2m}}{(-q^2;q^2)_{m+1}}=q^{-1} \sum_{m=1}^{\infty}\frac{q^{m^2}}{(-q^2;q^2)_m}=q^{-1}(\phi(q)-1).
\end{align*}
Next, note that from \cite[p.~29, Exercise 4]{gea}, we have
\begin{equation*}
\sum_{n=0}^{\infty}\frac{(a;q)_n(b;q^2)_nt^n}{(q;q)_n(atb;q^2)_n}=\frac{(at;q^2)_{\infty}(bt;q^2)_{\infty}}{(t;q^2)_{\infty}(abt;q^2)_{\infty}}\sum_{m=0}^{\infty}\frac{(a;q^2)_{m}(b;q^2)_{m}(tq)^{m}}{(q^2;q^2)_{m}(bt;q^2)_{m}}.
\end{equation*}
Now set $a\to -q^2/t, b\to q^2$ in the above equation to see that
\begin{align*}
\sum_{n=0}^{\infty} \frac{ (-q^2/t;q)_n (q^2;q^2)_n  t^n}{ (q;q)_n (-q^4;q^2)_n}=\frac{(-q^2;q^2)_{\infty} (q^2 t ;q^2)_{\infty}}{(t;q^2)_{\infty} (-q^4;q^2)_{\infty}} \sum_{n=0}^{\infty} \frac{ (-q^2/t;q^2)_n (q^2;q^2)_n (tq)^n}{(q^2;q^2)_{n} (q^2 t;q^2)_n}.
\end{align*}
Let $t\to 0$. Then
\begin{align*}
\sum_{n=0}^{\infty} \frac{(q^2;q^2)_n q^{\binom{n}{2}+2n}}{(q;q)_n (-q^2;q^2)_{n+1}}=\sum_{n=0}^{\infty} q^{n^2+2n}.
\end{align*}
Hence,
\begin{align*}
U_2(q)&=\sum_{m=0}^{\infty} \frac{(-q;q)_m  q^{\binom{m}{2}}  q^{2m}}{(-q^2;q^2)_{m+1}}\\
&=\sum_{n=0}^{\infty} \frac{(q^2;q^2)_n q^{\binom{n}{2}+2n}}{(q;q)_n (-q^2;q^2)_{n+1}}=\sum_{n=0}^{\infty} q^{n^2+2n}
\end{align*}
Finally,
\begin{align*}
q \sum_{n=0}^{\infty} q^n (-q^{n+1};q)_n (-q^{2n+1};q^2)_{\infty} &= \frac{-qU_1(q)}{(-q;q^2)_{\infty}} + 2 qU_2(q)\\
&=\frac{1-\phi(q)}{(-q;q^2)_{\infty}} +\sum_{n=-\infty}^{\infty} q^{n^2} -1\\
&=\frac{1-\phi(q)}{(-q;q^2)_{\infty}} +(q^2;q^2)_{\infty} (-q;q^2)_{\infty}^2 -1,
\end{align*}
where in the last step we used the Jacobi triple product identity \cite[p.~21, Theorem 2.8]{gea}.
\end{proof}

\section{Analogues of Euler's pentagonal number theorem}\label{eulerpenta}
We begin with a result, motivated from studying a generating function similar to that of $p_{\omega}(n)$, which has an interesting partition-theoretic interpretation analogous to that of Euler's pentagonal number theorem \cite[p.~10, Theorem 1.6]{gea}. This interpretation is provided after the proof of the following theorem.
\begin{theorem}
The following identity holds:
\begin{align*}
\sum_{n=1}^{\infty} \frac{q^n}{(-q^n;q)_{n+1}(-q^{2n+2};q^2)_{\infty}} =\sum_{j=0}^{\infty} (-1)^j q^{6j^2+4j+1}(1+q^{4j+2}).
\end{align*}
\end{theorem}

\proof
{\allowdisplaybreaks\begin{align*}
\sum_{n=1}^{\infty} \frac{q^n}{(-q^n;q)_{n+1} (-q^{2n+2};q^2)_{\infty}} &=\sum_{n=1}^{\infty} \frac{(-q;q)_{n-1} q^n}{(-q;q)_{2n} (-q^{2n+2};q^2)_{\infty}} \\
&=\frac{1}{(-q^2;q^2)_{\infty}} \sum_{n=0}^{\infty} \frac{(-q;q)_n q^{n+1} (-q^2;q^2)_{n+1}}{(-q;q)_{2n+2}}\\
&=\frac{q}{(1+q)(-q^2;q^2)_{\infty}} \sum_{n=0}^{\infty} \frac{(-q;q)_n q^n}{(-q^3;q^2)_n}\\
&=\frac{q}{(1+q)(-q^2;q^2)_{\infty}} \sum_{n=0}^{\infty} \frac{ q^n}{(q;q)_n}\frac{(q^2;q^2)_n}{(-q^3;q^2)_n}\\
&=\frac{q(q^2;q^2)_{\infty}}{(-q;q^2)_{\infty} (-q^2;q^2)_{\infty}} \sum_{n=0}^{\infty} \frac{ q^n}{(q;q)_n}\frac{(-q^{2n+3};q^2)_\infty}{(q^{2n+2};q^2)_\infty}\\
&=\frac{q(q^2;q^2)_{\infty}}{(-q;q)_{\infty}} \sum_{n=0}^{\infty} \frac{ q^n}{(q;q)_n}
\sum_{m=0}^{\infty} \frac{(-q;q^2)_m q^{m(2n+2)}}{(q^{2};q^2)_m}\\
&={q(q;q)_{\infty}} \sum_{m=0}^{\infty}  \frac{(-q;q^2)_m q^{2m}}{(q^{2};q^2)_m} \frac{1}{(q^{2m+1};q)_{\infty}}\\
&={q} \sum_{m=0}^{\infty}  \frac{(-q;q^2)_m (q;q)_{2m} q^{2m}}{(q^{2};q^2)_m} \\
&=q \sum_{m=0}^{\infty} (q^2;q^4)_m q^{2m}\\
&=\sum_{n=0}^{\infty} (-1)^n q^{6n^2+4n+1} (1+q^{4n+2}),
\end{align*}}
where we used \eqref{qbin} in the sixth step, \eqref{euleri} in the next, and finally Entry 9.5.2 from \cite[p.~238]{AB1} in the last step.  

\textbf{Remark.} One can also set $a=-q, b=q, t=q^2$ and $c\to 0$ in \cite[p.~67, Theorem $A_3$]{geaquart}, \cite[p.~229, Equation (9.3.3)]{AB1} so as to obtain
{\allowdisplaybreaks\begin{align*}
\sum_{n=0}^{\infty}\frac{(q^2;q^2)_nq^n}{(q;q)_n(-q^3;q^2)_n}&=\frac{(q^2;q^2)_{\infty}}{(q;q)_{\infty}(-q^3;q^2)_{\infty}}\sum_{n=0}^{\infty}\frac{(-q;q^2)_n(q;q)_{2n}q^{2n}}{(q^2;q^2)_n}\nonumber\\
&=\frac{(q^2;q^2)_{\infty}}{(q;q)_{\infty}(-q^3;q^2)_{\infty}}\sum_{n=0}^{\infty}(q^2;q^4)_nq^{2n}
\end{align*}}
directly.

One can conceive the above result through its interesting partition-theoretic interpretation given in the following theorem.

\begin{theorem}\label{pti1}
Among the partitions in which each odd part is less than twice the smallest part, let $d_o(n)$ be the number of such partitions with an odd number of parts and let $d_e(n)$ be the number of such partitions with an even number of parts. Then for $j\geq 0$,
\begin{equation*}
d_o(n)-d_e(n)=
\begin{cases}
(-1)^j,\hspace{2mm}\text{if}\hspace{2mm}n=6j^2+4j+1\hspace{2mm}\text{or}\hspace{2mm}n=6j^2+8j+3,\\
0,\hspace{2mm}\text{otherwise}. 
\end{cases}
\end{equation*}
\end{theorem}
For example, if $n=3$, the three partitions to be considered are $3, 2+1, 1+1+1$, and $d_o(3)=2$ and $d_e(3)=1$, where as 
if $n=4$, the four partitions are $4, 2+2, 2+1+1, 1+1+1+1$, and $d_o(4)=2=d_e(4)$.

There is another result that one can obtain by working analogously with a generating function similar to that of $p_{\nu}(n)$, namely,
\begin{theorem}
We have
\begin{align*}
\sum_{n=0}^{\infty} q^n (q^{n+1};q)_n (q^{2n+2};q^2)_{\infty} =\sum_{j=0 }^{\infty} (-1)^j q^{j(3j+2)} (1+q^{2j+1}).
\end{align*}
\end{theorem}

\proof
Starting with the left-hand side and using \eqref{euleri} in the third and the fifth step below, we see that
{\allowdisplaybreaks\begin{align*}
\sum_{n=0}^{\infty} q^n (q^{n+1};q)_n (q^{2n+2};q^2)_{\infty}  &=(q^2;q^2)_{\infty} \sum_{n=0}^{\infty} \frac{(q;q^2)_{n} q^n}{(q;q)_{n} } \\
&={(q^2;q^2)_{\infty} (q;q^2)_{\infty} } \sum_{n=0}^{\infty} \frac{q^{n} }{(q;q)_{n} (q^{2n+1};q^2)_{\infty} }\\
&=(q;q)_{\infty} \sum_{n=0}^{\infty} \frac{ q^n}{(q;q)_n} \sum_{m=0}^{\infty} \frac{q^{(2n+1)m}}{(q^2;q^2)_{m}} \\
&=(q;q)_{\infty} \sum_{m=0}^{\infty} \frac{q^m}{(q^2;q^2)_{m}} \sum_{n=0}^{\infty} \frac{q^{n(2m+1)}}{(q;q)_{n}}\\
&=(q;q)_{\infty} \sum_{m=0}^{\infty} \frac{q^m}{(q^2;q^2)_{m} (q^{2m+1};q)_{\infty}} \\
&=\sum_{m=0}^{\infty} \frac{(q;q)_{2m} q^m}{(q^2;q^2)_m}\\
&= \sum_{m=0}^{\infty} (q;q^2)_m q^{m}\\
&=\sum_{m=0}^{\infty} (-1)^m q^{3n^2+2n} (1+q^{2n+1}),
\end{align*}}
again by Entry 9.5.2 in \cite[p.~238]{AB1}.

\endproof
This result also has a partition-theoretic interpretation similar to that of Theorem \ref{pti1}.
\begin{theorem}\label{pti2}
Among the partitions into distinct parts such that each odd part is less than twice the smallest part, let $d_o(n)$ be the number of such partitions with an odd number of parts and let $d_e(n)$ be the number of such partitions with an even number of parts. Then for $j\geq 0$,
\begin{equation*}
d_o(n)-d_e(n)=
\begin{cases}
(-1)^j,\hspace{2mm}\text{if}\hspace{2mm}n=3j^2+2j\hspace{2mm}\text{or}\hspace{2mm}n=3j^2+4j+1,\\
0,\hspace{2mm}\text{otherwise}. 
\end{cases}
\end{equation*}
\end{theorem}

\textbf{Remark.} As a side observation, we note that a similar treatment, as above, of the left-hand side of Theorem \ref{phiqthm} leads to the following result:
\begin{align*}
1- q \sum_{n=0}^{\infty} q^n (q^{n+1};q)_n (q^{2n+1};q^2)_{\infty} = (q;q^2)_{\infty}.
\end{align*}
This follows immediately when we observe that
\begin{align*}
1- q \sum_{n=0}^{\infty} q^n (q^{n+1};q)_n (q^{2n+1};q^2)_{\infty} &= 1- q (q;q^2)_{\infty} \sum_{n=0}^{\infty} \frac{q^n (q^{n+1};q)_n(q^2;q^2)_{n}}{(q;q)_{2n}}\\
&=1-(q;q^2)_{\infty}\sum_{n=1}^{\infty}(-q;q)_{n-1}q^n,
\end{align*}
and then use \eqref{seplarge} and \eqref{ei}.
\section{Congruences}\label{cong} 

Two new partition functions, namely $p_{\omega}(n)$ and $p_{\nu}(n)$, are introduced here. For any positive integer $n$, $p_{\omega}(n)$ counts the number of partitions in which all odd parts are less than twice the smallest part, and $p_{\nu}(n)$ counts the number of partitions in which the parts are distinct and all odd parts are less than twice the smallest part. Then it follows that
\begin{align*}
\sum_{n=1}^{\infty} p_{\omega}(n) q^n &=\sum_{n=1}^{\infty} \frac{q^n}{(1-q^n) (q^{n+1};q)_{n} (q^{2n+2};q^2)_{\infty}},\\
\sum_{n=1}^{\infty}  p_{\nu} (n) q^n & = \sum_{n=1}^{\infty}  q^{n} (-q^{n+1};q)_n (-q^{2n+2};q^2)_{\infty}. 
\end{align*}
We now define two analogues of ${\rm spt}(n)$. Denote by ${\rm spt}_{\omega}(n)$ and ${\rm spt}_{\nu}(n)$ the number of smallest parts in the partitions of $n$ enumerated by $p_{\omega}(n)$ and $p_{\nu}(n)$ repsectively.  Note that
 since the parts in partitions counted by $p_{\nu}(n)$ are all distinct,  $\text{spt}_{\nu}(n)=p_{\nu}(n)$.  Also, from the definition of ${\rm spt}_{\omega}(n)$,
\begin{align*}
\sum_{n=1}^{\infty} {\rm spt}_{\omega}(n) q^n &=\sum_{n=1}^{\infty} \frac{q^n}{(1-q^n)^2 (q^{n+1};q)_{n} (q^{2n+2};q^2)_{\infty}}. 
\end{align*}
In the following lemma, we first derive an alternative representation of the above generating function which will help us obtain congruences for $\textup{spt}_{\omega}(n)$. 
\begin{lemma}\label{con1}
The following identity holds:
\begin{align}\label{baileydiff}
\sum_{n=1}^{\infty} \frac{q^n}{(1-q^n)^2 (q^{n+1};q)_n (q^{2n+2};q^2)_{\infty}} 
=\frac{1}{(q^2;q^2)_{\infty}} \sum_{n=1}^{\infty} \frac{nq^n}{1-q^n} +\frac{1}{(q^2;q^2)_{\infty}} \sum_{n=1}^{\infty} \frac{(-1)^n (1+q^{2n}) q^{n(3n+1)}}{(1-q^{2n})^2}.
\end{align}
\end{lemma}
\begin{proof}
Let $a=1, p_1=z=p_2^{-1}$ and then let $f\to\infty$ in Bailey's ${}_{10}\phi_{9}$ transformation, namely \eqref{baitra}, to see that
\begin{align*}
\sum_{n=0}^{\infty}\frac{(z;q)_n(z^{-1};q)_nq^n}{(q;q)_n(q;q^2)_n}=\frac{(zq;q)_{\infty}(z^{-1}q;q)_{\infty}}{(q;q)_{\infty}^2}\left(1+\sum_{n=1}^{\infty}\frac{(1-z)(1-z^{-1})(1+q^{2n})(-1)^nq^{n(3n+1)}}{(1-zq^{2n})(1-z^{-1}q^{2n})}\right).
\end{align*} 
Now take the second derivative on both sides with respect to $z$, make use of the facts \cite[Equations (2.1), (2.4)]{gea266}
\begin{align*}
-\frac{1}{2}\left[\frac{d^2}{dz^2}(1-z)(1-z^{-1})f(z)\right]_{z=1}&=f(1),\\
-\frac{1}{2}\left[\frac{d^2}{dz^2}(zq;q)_{\infty}(z^{-1}q;q)_{\infty}\right]_{z=1}&=(q;q)_{\infty}^{2}\sum_{n=1}^{\infty}\frac{nq^n}{1-q^n},
\end{align*}
to deduce that
\begin{align*}
\sum_{n=1}^{\infty}\frac{(q;q)_nq^n}{(q;q^2)_n(1-q^n)^2}=\sum_{n=1}^{\infty}\frac{nq^n}{1-q^n}+\sum_{n=1}^{\infty}\frac{(1+q^{2n})(-1)^nq^{n(3n+1)}}{(1-q^{2n})^2}.
\end{align*}
Finally, multiply both sides of the above identity by $1/(q^2;q^2)_{\infty}$ and simplify the resulting left-hand side to arrive at the claimed result.
\end{proof}
We now use the above lemma to prove the next theorem.
\begin{theorem}\label{congsptomega1}
The following congruence holds:
\begin{align*}
\textup{spt}_{\omega} (5n+3) &\equiv 0 \pmod{5}.
\end{align*}
\end{theorem} 
\begin{proof}
In \cite[Equation (3.4)]{gea266}, the second expression on the right hand side of \eqref{baileydiff} was shown to be 
\begin{align*}
-\frac{1}{2} \sum_{n=0}^{\infty} N_2(n) q^{2n}.
\end{align*}
Additionally \cite[p.~139]{gea266}
\begin{align*}
N_2(n)\equiv 0 \pmod{5}
\end{align*}
if $n\equiv 4 \pmod{5}$, and  equivalently if $2n \equiv 3 \pmod{5}.$
Hence, to prove that 
\begin{align*}
\text{spt}_{\omega} (5n+3) \equiv 0 \pmod{5},
\end{align*}
it suffices to show that the coefficients of $q^{5n+3}$ in 
\begin{align*}
\frac{1}{(q^2;q^2)_{\infty}} \sum_{n=1}^{\infty} \frac{nq^n}{1-q^n}
\end{align*}
are divisible by $5$. 
This is one of the objectives of the following theorem.
\end{proof}

\begin{theorem}
Let 
\begin{align*}
\sum_{n=0}^{\infty} c_n q^n := \frac{1}{(q^2;q^2)_{\infty}} \sum_{n=1}^{\infty} \frac{nq^n}{1-q^n}.
\end{align*}
Then $5\mid c_{5n+3}$ and $5\mid c_{5n+4}$.
\end{theorem}

\proof
Using Jacobi's identity for $(q;q)_{\infty}^3$ \cite[p.~176]{gea}, we see that
\begin{align*}
\sum_{n=0}^{\infty} (-1)^n (2n+1) \binom{n+1}{2} q^{\binom{n+1}{2}}&=q\frac{d}{dq} (q;q)_{\infty}^3\\
&=-3 (q;q)_{\infty}^3 \sum_{n=1}^{\infty} \frac{nq^n}{1-q^n}.
\end{align*}
Hence
\begin{align*}
\frac{1}{(q^2;q^2)_{\infty}} \sum_{n=1}^{\infty} \frac{nq^n}{1-q^n} &= \frac{1}{(q;q)_{\infty} (-q;q)_{\infty}} \left(\frac{-1}{3(q;q)_{\infty}^3} \sum_{j=0}^{\infty} (-1)^j (2j+1) \binom{j+1}{2} q^{\binom{j+1}{2}}\right)\\
&=\frac{-1}{3 (q;q)_{\infty}^5} \sum_{n=-\infty}^{\infty} (-1)^n q^{n^2} \sum_{j=0}^{\infty} (-1)^j (2j+1) \binom{j+1}{2} q^{\binom{j+1}{2}}\\
&\equiv \frac{-1}{3(q^5;q^5)_{\infty}}\sum_{n=-\infty}^{\infty} (-1)^n q^{n^2} \sum_{j=0}^{\infty} (-1)^j (2j+1) \binom{j+1}{2} q^{\binom{j+1}{2}}  \pmod{5},
\end{align*}
where in the second step we used \cite[p.~23, Equation (2.2.12)]{gea}.
So we need to know when is 
\begin{align*}
n^2+\binom{j+1}{2}\equiv 3 \pmod{5}.
\end{align*}
Since $n^2\equiv 0,1,$ or $4 \pmod{5}$, and $\binom{j+1}{2}\equiv 0,1$ or $3\pmod{5}$, the only way we can get $3$ mod $5$ is for $n^2\equiv 0$ and $\binom{j+1}{2}\equiv 3 \pmod{5}$. Now
$\binom{j+1}{2} \equiv 3 \pmod{5}$ if and only if $j\equiv 2 \pmod{5}$ and then 
\begin{align*}
(2j+1)\binom{j+1}{2} \equiv 0 \pmod{5}.
\end{align*}
Hence $5\mid c_{5n+3}$.

Second, we need to see when is
\begin{align*}
n^2+\binom{j+1}{2}\equiv 4 \pmod{5}.
\end{align*}
There are two possibilities here - either $n^2\equiv 1$ and $\binom{j+1}{2} \equiv 3 \pmod{5}$ or $n^2\equiv 4 \pmod{5}$ and $\binom{j+1}{2} \equiv 0 \pmod{5}$. The first one requires $j\equiv 2 \pmod{5}$. Consequently in both,
\begin{align*}
(2j+1) \binom{j+1}{2} \equiv 0 \pmod{5}.
\end{align*}
This proves that $5\mid c_{5n+4}$.
\endproof
We now prove two other congruences satisfied by $\textup{spt}_{\omega}(n)$.
\begin{theorem}\label{congsptomega2}
The following congruences hold:
\begin{align*}
\textup{spt}_{\omega} (10n+7) &\equiv 0 \pmod{5},\nonumber\\
\textup{spt}_{\omega} (10n+9) &\equiv 0 \pmod{5}.
\end{align*}
\end{theorem}
\begin{proof}
We again use \eqref{baileydiff} to prove the congruences. Since $2n\not\equiv 7, 9\hspace{1mm}(\text{mod}\hspace{1mm}10)$ for any $n$, there is no contribution 
from the generating function of $-\frac{1}{2}N_{2}(n)$. Hence our objective is only to prove that the coefficients of $q^{10n+7}$ and $q^{10n+9}$ in
\begin{align*}
\frac{1}{(q^2;q^2)_{\infty}} \sum_{n=1}^{\infty} \frac{nq^n}{1-q^n}
\end{align*}
are divisible by $5$. Since $10n+7$ and $10n+9$ are both odd, we only need to consider
\begin{align*}
\frac{1}{(q^2;q^2)_{\infty}} \sum_{n=0}^{\infty} \frac{(2n+1)q^{2n+1}}{1-q^{2n+1}}=: \sum_{n=0}^{\infty} d_n q^n.
\end{align*}
Now in  \cite[p.~79, Equation (32.31)]{fine}, we find
\begin{align*}
\sum_{n=0}^{\infty} \frac{(2n+1) q^{2n+1}}{1-q^{2n+1}} =\frac{q(q^4;q^4)^8_{\infty}}{(q^2;q^2)_{\infty}^4}.
\end{align*}
Hence
\begin{align*}
\sum_{n=0}^{\infty} d_n q^n &=q\left(\frac{(q^4;q^4)_{\infty}}{(q^2;q^2)_{\infty}}\right)^5 (q^4;q^4)_{\infty}^3\\
& \equiv q \frac{(q^{20};q^{20})_{\infty}}{(q^{10};q^{10})_{\infty}} \sum_{n=0}^{\infty} (-1)^n (2n+1) q^{2n^2+2n} \pmod{5}
\end{align*}
Now mod $5$, 
\begin{align*}
\begin{array}{c|c}
n & 2n^2+2n+1\\ \hline
0 & 1\\
1 & 0\\
2& 3\\
3 & 0\\
4& 1
\end{array}
\end{align*}
Thus we immediately see that $d_{5n+2}$ and $d_{5n+4}$ are divisible by $5$. Finally also $d_{5n+3}$ is divisible by $5$  because $3 \pmod{5}$ arises only for $n\equiv 2 \pmod{5}$, but then $(2n+1)\equiv 0 \pmod{5}$. Thus we have proved:

\begin{theorem}
$5 \mid d_n$ for $n\equiv 2, 3, 4 \pmod{5}.$
\end{theorem}

\begin{corollary} 
$5\mid d_{10n+7}, 5 \mid d_{10n+3}, 5 \mid d_{10n+9}$.
\end{corollary}

This then proves, in particular, that $5\mid \text{spt}_{\omega} (10n+7)$ and $5\mid \text{spt}_{\omega} (10n+9)$.
\end{proof}

Our last congruence is for the smallest parts function associated with $p_{\nu}(n)$, and is given in the following theorem.
\begin{theorem}
We have
\begin{align*}
{\rm spt}_{\nu}(10n+8)\equiv 0 \pmod{5}. 
\end{align*}
\end{theorem}

\begin{proof}
Since ${\rm spt}_{\nu}(n)$ equals $p_{\nu}(n)$ and thus ${\rm spt}_{\nu}(10n+8)$ is the coeffcient of $q^{10n+8}$ in the generating function 
\begin{equation*}
\sum_{n=0}^{\infty} q^n (-q^{n+1};q)_n (-q^{2n+2};q^2)_{\infty},
\end{equation*}
in view of \eqref{thm1.1}, it suffices to consider the contribution from $(-q^2;q^2)_{\infty} \psi(q^2)$ only. Now
\begin{align*}
(-q^2;q^2)_{\infty} \psi(q^2)&=(-q^2;q^2)_{\infty} \frac{(q^4;q^4)_{\infty}}{(q^2;q^4)_{\infty}}\\
&=(-q^2;q^2)_{\infty}^3 (q^2;q^2)_{\infty}\\
&=\frac{(-q^2;q^2)_{\infty}^3 (q^2;q^2)_{\infty}^6}{(q^2;q^2)_{\infty}^5}\\
&=\frac{(q^4;q^4)_{\infty}^3 (q^2;q^2)_{\infty}^3}{(q^2;q^2)_{\infty}^5}\\
&\equiv \frac{(q^4;q^4)_{\infty}^3 (q^2;q^2)_{\infty}^3}{(q^{10};q^{10})_{\infty}} \pmod{5}
\end{align*}
By Jacobi's identity, 
\begin{align*}
(q^4;q^4)_{\infty}^3 (q^2;q^2)_{\infty}^3=\sum_{m=0}^{\infty} (-1)^m (2m+1)q^{2m(m+1)} \sum_{j=0}^{\infty} (-1)^j (2j+1)q^{j(j+1)}.
\end{align*}
Since
$j(j+1)\equiv 0, 2, 6 \pmod{10}$ and $2m(m+1)\equiv 0, 4, 2 \pmod{10}$, we get 
\begin{align*}
j(j+1)+2m(m+1)\equiv 8 \pmod{10}  &\quad  \text{ iff } j(j+1)\equiv 6, 2m(m+1)\equiv 2 \pmod{10}\\
& \quad  \text{ iff } j\equiv 2, m\equiv 2 \pmod{10}.
\end{align*}
So for $j,m \equiv 2 \pmod{10}$,
\begin{align*}
(2m+1)(2j+1)\equiv 0 \pmod{5},
\end{align*}
which completes the proof.
\end{proof}
As a result of the above theorem, we obtain the following corollary. 
\begin{corollary}
The coefficient of $q^{10n+8}$ in $\nu(-q)$  is divisible by $5$. 
\end{corollary}

\section{Some results on generalized third order mock theta functions}\label{genmock}

In \cite[p.~78]{geaquart}, Andrews gave generalizations of five of the seven third order mock theta functions, and proved that when $\a=q^r, r\in\mathbb{Z}^{+}$,  they are indeed mock theta functions. In particular, he defined
\begin{align*}
\omega(\a;q)&:=\sum_{n=0}^{\infty}\frac{q^{2n^2}\a^{2n}}{(q;q^2)_{n+1}(\a^2q^{-1};q^2)_{n+1}},\\
\nu(\a;q)&:=\sum_{n=0}^{\infty}\frac{q^{n(n+1)}}{(-\a^2q^{-1};q^2)_{n+1}}.
\end{align*}
It is easy to see that for $\a=q$, the above functions are equal to $\omega(q)$ and $\nu(q)$ defined in \eqref{omegaq} and \eqref{nuq} respectively. In this section, we give  generalizations of two of the results from previous sections. For example, the following theorem holds:
\begin{theorem}\label{qomegaqgen}
For $|\a^2|<|q|<1$, we have
\begin{align}\label{qomegaqgeneq}
\a^4q^{-2}\omega(\a^2;q^2)&=\frac{\a^2(q^6;q^4)_{\infty}}{(\a^2;q^2)_{\infty}}+\frac{\a^2(q^6;q^4)_{\infty}(q^4;q^4)_{\infty}}{(\a^2q^2;q^2)_{\infty}}\sum_{n=2}^{\infty}\frac{q^{2n-2}(\a^2q^2;q^2)_{n-2}}{(q^{4};q^2)_{2n-1}(q^{4n+4};q^4)_{\infty}}\nonumber\\
&\quad-\frac{\a^2q^{-1}(q^2;q^4)_{\infty}}{2(1+\a^2q^{-1})(\a^2;q^2)_{\infty}}\left(\sum_{n=0}^{\infty}-\sum_{n=-\infty}^{-1}\right)\frac{(-q;q^2)_{n}q^n}{(-\a^2q;q^2)_n}.
\end{align}
\end{theorem}
When $\a=q$, we obtain Theorem \ref{qomegaq} with $q$ replaced by $q^2$. It may be worthwhile seeing if the above theorem has an interesting partition-theoretic interpretation. 

Similarly, for $|\a^2|<|q|<1$, one has 
\begin{align}\label{genalpha}
\frac{\nu(-\a;-q)}{(-q^2;q^2)_{\infty}}&=\frac{-1}{(-q^2;q^2)_{\infty}}\sum_{m=0}^{\infty}(\a^{-2}q^{3};q^2)_m(-\a^2q^{-1})^{m}\nonumber\\
&\quad+\frac{2(q^2;q^2)_{\infty}}{(-q;q)_{\infty}(\a^2;q^2)_{\infty}(\a^2q^{-1};q^2)_{\infty}}\sum_{m=0}^{\infty}\frac{(\a^2q^{-1};q)_{m}(-q;q)_{m}q^{m(m+1)/2}}{(q;q)_{m}(-\a^2q^{-1};q^2)_{m+1}}.
\end{align}
When $\a=q$, one gets
\begin{equation*}
\frac{\nu(-q)}{(-q^2;q^2)_{\infty}}=\frac{-S_1(q)}{(-q^2;q^2)_{\infty}}+2S_2(q),
\end{equation*}
which is the identity obtained from \eqref{vi} and \eqref{nufinal}. It remains to be seen if a full generalization of Theorem \ref{nuqthm} for $\nu(-\a;-q)$ exists. 

We refrain ourselves from giving proofs of \eqref{qomegaqgeneq} and \eqref{genalpha} since they involve similar ideas as those involved in their special cases proved in this paper. We only mention, however, that one starts the proofs by using the following identity from \cite[p.~78, Equation (3c)]{geaquart}:
\begin{equation*}
\nu(-\a;-q)=\frac{(q^2;q^2)_{\infty}(-q^2;q^2)_{\infty}^{2}}{(\a^4q^{-2};q^4)_{\infty}}+\a^2q^{-1}\omega(\a^2;q^2),
\end{equation*}
which gives \eqref{thm1.1} as a special case when $\a=q$.

\begin{center}
\textbf{Acknowledgements}

\end{center}
The second author was funded in part by the grant NSF-DMS 1112656 of
Professor Victor H.~Moll of Tulane University, whom he sincerely thanks for
this support. The third author was partially supported by a grant ($\#280903$) from the Simons Foundation.

\end{document}